\documentclass{article}
\usepackage[T1]{fontenc}
\usepackage{a4wide}
\usepackage{lipsum}
\usepackage{hyperref}

\usepackage{todonotes}

\title{On 3-Connected Planar Graphs with Unique Orientable Circuit Double Covers}

\author{ 
Meike Weiß\thanks{ RWTH Aachen University, Aachen, Germany. E-mail: {\tt weiss@art.rwth-aachen.de}.}
\and
Reymond Akpanya\thanks{The University of Sydney, Sydney, Australia. E-mail: {\tt reymond.akpanya@sydney.edu.au}.}
\and
Alice C.\ Niemeyer\thanks{RWTH Aachen University, Germany. E-mail: {\tt alice.niemeyer@art.rwth-aachen.de}.}
}
\date{}

\makeatletter

\usepackage{amsthm,amsmath,amssymb}
\usepackage{float}
\usepackage{cleveref}
\usepackage{comment}
\usepackage{cleveref}

\newtheorem{theorem}{Theorem}[section]
\newtheorem*{theorem*}{Theorem}
\newtheorem{lemma}[theorem]{Lemma}
\newtheorem{definition}[theorem]{Definition}

\newtheorem{corollary}[theorem]{Corollary}
\newtheorem*{corollary*}{Corollary}
\newtheorem*{conjecture*}{Conjecture}

\newtheorem{remark}[theorem]{Remark}
\newtheorem{proposition}[theorem]{Proposition}
\newtheorem*{question*}{Question}
\usepackage{subcaption}

\usepackage{amsmath, amssymb, amsfonts} 

\usepackage{ifthen}

\usepackage{tikz}
\usetikzlibrary{calc}
\usetikzlibrary{positioning}
\usetikzlibrary{shapes}
\usetikzlibrary{patterns}
\usepackage{circuitikz}

\makeatletter
\edef\texforht{TT\noexpand\fi
  \@ifpackageloaded{tex4ht}
    {\noexpand\iftrue}
    {\noexpand\iffalse}}
\makeatother

\makeatletter
\newif\iftikz@node@phantom
\tikzset{
  phantom/.is if=tikz@node@phantom,
  text/.code=%
    \edef\tikz@temp{#1}%
    \ifx\tikz@temp\tikz@nonetext
      \tikz@node@phantomtrue
    \else
      \tikz@node@phantomfalse
      \let\tikz@textcolor\tikz@temp
    \fi
}
\usepackage{etoolbox}
\patchcmd\tikz@fig@continue{\tikz@node@transformations}{%
  \iftikz@node@phantom
    \setbox\pgfnodeparttextbox\hbox{}
  \fi\tikz@node@transformations}{}{}
\makeatother

\newcommand{\tikzAngleOfLine}{\tikz@AngleOfLine}
\def\tikz@AngleOfLine(#1)(#2)#3{%
  \pgfmathanglebetweenpoints{%
    \pgfpointanchor{#1}{center}}{%
    \pgfpointanchor{#2}{center}}
  \pgfmathsetmacro{#3}{\pgfmathresult}%
}

\tikzset{ 
    vertexNodePlain/.style = {fill=none, shape=circle, inner sep=0pt, minimum size=2pt, text=none},
    vertexNodePlain/.default=white,
    vertexPlain/labels/.style = {
        vertexNode/.style={vertexNodePlain=##1},
        vertexLabel/.style={gray}
    },
    vertexPlain/nolabels/.style = {
        vertexNode/.style={vertexNodePlain=##1},
        vertexLabel/.style={text=none}
    },
    vertexPlain/.style = vertexPlain/#1,
    vertexPlain/.default=labels
}
\tikzset{
    vertexNodeNormal/.style = {fill=none, shape=circle, inner sep=0pt, minimum size=4pt, text=none},
    vertexNodeNormal/.default = blue,
    vertexNormal/labels/.style = {
        vertexNode/.style={vertexNodeNormal=##1},
        vertexLabel/.style={blue}
    },
    vertexNormal/nolabels/.style = {
        vertexNode/.style={vertexNodeNormal=##1},
        vertexLabel/.style={text=none}
    },
    vertexNormal/.style = vertexNormal/#1,
    vertexNormal/.default=labels
}
\tikzset{
    vertexNodeBallShading/pdf/.style = {ball color=#1},
    vertexNodeBallShading/svg/.style = {fill=#1},
    vertexNodeBallShading/.code = {
        \if\texforht
            \tikzset{vertexNodeBallShading/svg=white}
        \else
            \tikzset{vertexNodeBallShading/pdf=white}
        \fi
    },
    vertexNodeBall/.style = {shape=circle, vertexNodeBallShading=#1, inner sep=2pt, outer sep=0pt, minimum size=3pt, font=\tiny},
    vertexNodeBall/.default = white,
    vertexBall/labels/.style = {
        vertexNode/.style={vertexNodeBall=##1, text=black},
        vertexLabel/.style={text=none}
    },
    vertexBall/nolabels/.style = {
        vertexNode/.style={vertexNodeBall=##1, text=none},
        vertexLabel/.style={text=none}
    },
    vertexBall/.style = vertexBall/#1,
    vertexBall/.default=labels
}
\tikzset{ 
    vertexStyle/.style={vertexNormal=#1},
    vertexStyle/.default = labels
}

\newcommand{\vertexLabelR}[4][]{
    \ifthenelse{ \equal{#1}{} }
        { \node[vertexNode] at (#2) {#4}; }
        { \node[vertexNode=#1] at (#2) {#4}; }
    \node[vertexLabel, #3] at (#2) {#4};
}
\newcommand{\vertexLabelA}[4][]{
    \ifthenelse{ \equal{#1}{} }
        { \node[vertexNode] at (#2) {#4}; }
        { \node[vertexNode=#1] at (#2) {#4}; }
    \node[vertexLabel] at (#3) {#4};
}

\newcommand{\edgeLabelColor}{blue!20!white}
\tikzset{
    edgeLineNone/.style = {draw=none},
    edgeLineNone/.default=black,
    edgeNone/labels/.style = {
        edge/.style = {edgeLineNone=##1},
        edgeLabel/.style = {fill=\edgeLabelColor,font=\small}
    },
    edgeNone/nolabels/.style = {
        edge/.style = {edgeLineNone=##1},
        edgeLabel/.style = {text=none}
    },
    edgeNone/.style = edgeNone/#1,
    edgeNone/.default = labels
}
\tikzset{
    edgeLinePlain/.style={line join=round, draw=#1},
    edgeLinePlain/.default=black,
    edgePlain/labels/.style = {
        edge/.style={edgeLinePlain=##1},
        edgeLabel/.style={fill=\edgeLabelColor,font=\small}
    },
    edgePlain/nolabels/.style = {
        edge/.style={edgeLinePlain=##1},
        edgeLabel/.style={text=none}
    },
    edgePlain/.style = edgePlain/#1,
    edgePlain/.default = labels
}
\tikzset{
    edgeLineDouble/.style = {very thin, double=#1, double distance=.8pt, line join=round},
    edgeLineDouble/.default=gray!90!white,
    edgeDouble/labels/.style = {
        edge/.style = {edgeLineDouble=##1},
        edgeLabel/.style = {fill=\edgeLabelColor,font=\small}
    },
    edgeDouble/nolabels/.style = {
        edge/.style = {edgeLineDouble=##1},
        edgeLabel/.style = {text=none}
    },
    edgeDouble/.style = edgeDouble/#1,
    edgeDouble/.default = labels
}
\tikzset{
    edgeStyle/.style = {edgePlain=#1},
    edgeStyle/.default = labels
}

\newcommand{\faceColorY}{yellow!60!white}   
\newcommand{\faceColorB}{blue!60!white}     
\newcommand{\faceColorC}{cyan!60}           
\newcommand{\faceColorR}{red!60!white}      
\newcommand{\faceColorG}{green!60!white}    
\newcommand{\faceColorO}{orange!50!yellow!70!white} 

\newcommand{\faceColor}{\faceColorY}
\newcommand{\faceColorSwap}{\faceColorC}

\tikzset{
    face/.style = {fill=#1},
    face/.default = \faceColor,
    faceY/.style = {face=\faceColorY},
    faceB/.style = {face=\faceColorB},
    faceC/.style = {face=\faceColorC},
    faceR/.style = {face=\faceColorR},
    faceG/.style = {face=\faceColorG},
    faceO/.style = {face=\faceColorO}
}
\tikzset{
    faceStyle/labels/.style = {
        faceLabel/.style = {}
    },
    faceStyle/nolabels/.style = {
        faceLabel/.style = {text=none}
    },
    faceStyle/.style = faceStyle/#1,
    faceStyle/.default = labels
}
\tikzset{ face/.style={fill=#1} }
\tikzset{ faceSwap/.code=
    \ifdefined\swapColors
        \tikzset{face=\faceColorSwap}
    \else
        \tikzset{face=\faceColor}
    \fi
}

\begin{document}

\thispagestyle{empty}
\maketitle

\begin{abstract}
A circuit double cover of a bridgeless graph is a collection of even subgraphs such that every edge is contained in exactly two subgraphs of the given collection. Such a circuit double cover describes an embedding of the corresponding graph onto a surface. In this paper, we investigate the well-known Orientable Strong Embedding Conjecture. 
This conjecture proposes that every bridgeless graph has a circuit double cover describing an embedding on an orientable surface. In a recent paper, we have proved that a $3$-connected cubic planar graph $G$ has exactly one orientable circuit double cover if and only if $G$ is the dual graph of an Apollonian network. In this paper, we extend this result by demonstrating that this characterisation applies to any 3-connected planar graph, regardless of whether it is cubic.
\end{abstract}

\section{Introduction}
In 1985, F.\ Jaeger posed the famous \textbf{Orientable Strong Embedding Conjecture}, see~\cite{CycleDoubleCoverConjecture}.
This conjecture asserts that every $2$-connected graph has a strong embedding on some orientable surface. Although the conjecture remains unsolved, various partial results on this and related problems have been established over the years. 
For some results the reader is referred to~\cite{ELLINGHAM2011495,huseksamal,HAGGKVIST2006183,NEGAMI1988276,ZHA1995259}. The above conjecture has led to different investigations of embeddings of planar graphs on non-spherical surfaces. For instance, \cite{richter} establishes
that every $3$-connected planar graph has a strong embedding on some non-spherical surface. 
In a recent paper we have investigated the \textbf{Apollonian duals}, i.e.\ the dual graphs of Apollonian networks in the context of the Orientable Strong Embedding Conjecture, see~\cite{UnserPaper}. Here, an \textbf{Apollonian network} is a planar graph that can be constructed from recursively subdividing the faces of the complete graph $K_4$ into three new faces, see ~\cite{apolloniannetwork_Intro,fowler,grünbaum} for some references.
We have exploited the results on re-embeddings of $3$-connected cubic planar graphs established in \cite{EnamiEmbeddings,PaperMeikeStrong} to show that Apollonian duals are exactly the $3$-connected cubic planar graphs with exactly one orientable strong embedding. Given the fact that all 3-connected cubic planar graphs with a unique strong orientable embedding are therefore fully classified, the following question naturally arises:
\begin{question*}
 Which $3$-connected planar graphs have exactly one orientable strong embedding?
\end{question*}
A strong embedding of a given bridgeless graph can be described by a circuit double cover. In our work, a circuit denotes a connected subgraph in which every vertex has even degree.
 Hence, if a cubic graph has a circuit double cover, the circuits of the cover have to be cycles, i.e.\ connected $2$-regular subgraphs.
These observations allow us to extend \cite[Theorem 5.4]{UnserPaper} in the following way:
\begin{theorem}\label{theorem:main}
    A $3$-connected planar graph $G$ has exactly one orientable circuit double cover if and only if $G$ is an Apollonian dual.
\end{theorem}

Our paper is structured as follows: In \Cref{sec:preliminaries} we introduce notions on graphs and their embeddings which are essential for this work.
Two modifications of $3$-connected planar graphs, namely the complete augmentation (\Cref{def:completeaug}) and the complete truncation (\Cref{def:completetruncation}) are presented in \Cref{sec:modification}. Since a graph resulting from one of the above modifications is still $3$-connected and planar, these modifications play a central role in this work. We make use of them to prove the main result of this work.
In \Cref{sec:orientable}, we employ these modifications and demonstrate that they are well-behaved under the dualisation of a $3$-connected planar graph in the following sense: We show that the graph resulting from the complete augmentation of the dual of a $3$-connected planar graph $G$ is isomorphic to the dual of the graph resulting from the complete truncation of $G$, see \Cref{fig:proofidea} for an illustration of the above statement. We use this statement to show that the Apollonian duals are exactly the $3$-connected planar graphs with exactly one circuit double cover. Note that we investigated Apollonian networks and their properties using the computer algebra systems GAP \cite{GAP4} and Magma \cite{magma}.

\section{Preliminaries}\label{sec:preliminaries}
We start this work by introducing some preliminary notions on graphs.
For a more detailed description of the theoretical background needed for this paper, we refer the reader to \cite[Section~2]{UnserPaper}. Here, we assume that all the graphs in this paper are undirected, connected, simple and finite. Let $G$ be such a graph. We denote the vertex and edge set of $G$ by $V(G)$ and $E(G)$, respectively. Next, we formalise the notions of circuits and cycles in a graph. While their definitions vary across the literature, we adopt the following conventions in this work:
A \textbf{circuit} in $G$ is an even subgraph of $G$. In the case that a given circuit in $G$ is a 2-regular graph, it is called a \textbf{cycle} in $G$. A \textbf{triangle} in $G$ is a cycle of length three. We refer to a triangle $(v_1,v_2,v_3)$ of $G$ as \textbf{separating}, if $G\setminus\{v_1,v_2,v_3\}$ is no longer connected.
Moreover, a \textbf{circuit double cover} of $G$ is a set of circuits in $G$ such that each edge of $E(G)$ is contained in exactly two of these circuits. 
Such a circuit double cover is called \textbf{orientable} if there exists an orientation for each of the circuits in the cover, such that for every edge $e$ the two circuits covering $e$ are oriented in opposite directions through $e$. A \textbf{cycle double cover} and an \textbf{orientable cycle double cover} are defined analogously by replacing circuits with cycles.
Note that, in a cubic graph, all circuits are cycles, since all vertices have a degree of three.
In order to give an example of a circuit double cover that does not form a cycle double cover, we consider the graph $K_{2,2,2}.$
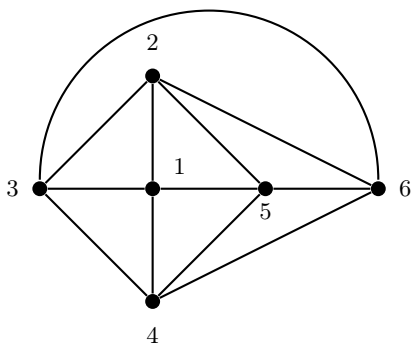
\begin{figure}[H]
    \centering
    \begin{figure}[H]
    \centering
    \begin{tikzpicture}[scale=3]
        \tikzset{knoten/.style={circle,fill=black,inner sep=0.7mm}}

        \node [knoten] (a) at (0,0) {};        
        \node [knoten] (b) at (0,0.5) {};      
        \node [knoten] (c) at (-0.5,0) {};     
        \node [knoten] (d) at (0,-0.5) {}; 
        \node [knoten] (e) at (0.5,0) {};      
        \node [knoten] (f) at (1,0) {};        

        \node at (0.12,0.1) {\small 1};
        \node at (0,0.65) {\small 2};
        \node at (-0.62,0) {\small 3};
        \node at (0,-0.65) {\small 4};
        \node at (0.5,-0.1) {\small 5};
        \node at (1.12,0) {\small 6};

        \draw[-,thick] (a) -- (b);
        \draw[-,thick] (a) -- (c);
        \draw[-,thick] (a) -- (d);
        \draw[-,thick] (a) -- (e);
        \draw[-,thick] (b) -- (c);
        \draw[-,thick] (b) -- (e);
        \draw[-,thick] (c) -- (d);
        \draw[-,thick] (d) -- (e);

        \draw[-,thick] (f) -- (b);
        \draw[-,thick] (f) -- (d);
        \draw[-,thick] (f) -- (e);

        \draw[-,thick] (-0.5,0.04) arc (180:0:0.75cm);
    \end{tikzpicture}
    \caption{A planar embedding of $K_{2,2,2}$}
\end{figure}
    \label{fig:placeholder}
\end{figure}
This graph can be equipped with the following circuit double cover:
\begin{align*}
    \{(1,2,3,1,4,5),(1,2,5,1,3,4),(6,2,3,6,4,5),(6,2,5,6,3,4)\}.
\end{align*}
We observe that the circuits of the above set do not form cycles of $K_{2,2,2}$.

An \textbf{embedding} of a graph $G$ on a compact 2-dimensional manifold $S$ without boundary is an injective and continuous map $\beta :G\rightarrow S$. The \textbf{cells} of the embedding are the connected components of $S\setminus\beta(G)$ and the embedding $\beta$ is called \textbf{strong} if all cells are bounded by circuits. Circuit double covers define graph embeddings, whereas cycle double cover define strong graph embeddings. If the double covers are orientable, they define embeddings on orientable surfaces.
A graph $G$ is called \textbf{planar} if $G$ can be embedded on the plane. The cells of such a planar embedding of $G$ are called \textbf{faces} and denoted by $F(G)$. A face $F\in F(G)$ is denoted by $(v_1,\dots,v_k)$ if $F$ is bounded by the closed walk $(v_1,\dots,v_k)$. We furthermore say that $F$ has length $k$ and denote its length by $\vert F\vert$. Note that 3-connected planar graphs have a unique embedding on the plane, see \cite{Whitney}. A planar graph where all faces are triangles is called a \textbf{planar triangulation}.
From a combinatorial perspective, every embedding of $G$ on the plane can be defined by a \textbf{rotation system}, which defines for each vertex $v$ of $G$ a cyclic ordering of the edges that are incident to $v$. In order to describe embeddings of a planar graph on non-spherical surfaces, a signature $\lambda: E(G)\to \{-1,1\}$ has to be defined additionally. More details on rotation systems and signatures of (planar) graphs can be found in \cite{TopologicalGraphTheory,GraphsOnSurfaces}.

\section{Modifying planar graphs}\label{sec:modification}
In this section, we discuss the complete truncation and the complete augmentation of $3$-connected planar graphs.
These modifications will be used to prove that Apollonian duals are exactly the $3$-connected planar graphs with exactly one orientable circuit double cover, see \Cref{theorem:main}. In order to conclude this result we show that the mentioned modifications preserve the planarity and the $3$-connectivity of a given graph.

\subsection{Augmentation}
We begin this section by formally introducing the augmentation of a $3$-connected planar graph at a face of the given graph, see for example \cite{MR370327} for more details.

\begin{definition}\label{def:augmentation}
    Let $G$ be a 3-connected planar graph and $F=(v_1,\ldots,v_k)$ a face of $G$. We define a graph $G'$ via $V(G'):=V(G)\cup \{v_F\}$ and $E(G'):=(E(G)\cup \{\{v_{F},v_{i}\}\mid i=1,\ldots,k\}.$
    We say that $G'$ is obtained by \textbf{\emph{augmenting}} the face $F$ in $G$. 
\end{definition}

Note that if $G$ is the planar graph illustrated in \Cref{fig:aug_a}, then $G$ has exactly two faces of length~$6$. Augmenting that face of length $6$ which is not the outer face results in the graph illustrated in \Cref{fig:aug_b}.
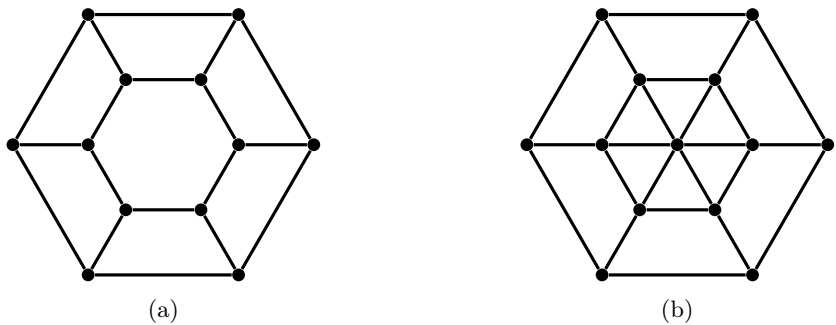
\begin{figure}[H]
    \begin{subfigure}{.45\textwidth}
       \centering
    \begin{tikzpicture}[vertexBall, edgeDouble, faceStyle, scale=2]

    \tikzset{knoten/.style={circle,fill=black,inner sep=0.6mm}}
    \node [knoten] (V1_1) at (1,0) {};
    \node [knoten] (V2_1) at (0.4999999999999999, 0.8660254037844386) {};
    \node [knoten] (V3_1) at (-0.5, 0.8660254037844388) {};
    \node [knoten] (V4_1) at (-1, 0.) {};
    \node [knoten] (V5_1) at (-.5, -0.8660254037844386) {};
    \node [knoten] (V6_1) at (0.5000000000000001, -0.8660254037844386) {};

    \node [knoten] (V1_2) at (0.5*1,0) {};
    \node [knoten] (V2_2) at (0.5*0.4999999999999999, 0.5*0.8660254037844386) {};
    \node [knoten] (V3_2) at (0.5*-0.5, 0.5*0.8660254037844388) {};
    \node [knoten] (V4_2) at (0.5*-1, 0.) {};
    \node [knoten] (V5_2) at (0.5*-.5, 0.5*-0.8660254037844386) {};
    \node [knoten] (V6_2) at (0.5*0.5000000000000001, 0.5*-0.8660254037844386) {};

    \draw[-, very thick] (V1_1) to (V1_2);
    \draw[-, very thick] (V2_1) to (V2_2);
    \draw[-, very thick] (V3_1) to (V3_2);
 \draw[-, very thick] (V4_1) to (V4_2);
    \draw[-, very thick] (V5_1) to (V5_2);
    \draw[-, very thick] (V6_1) to (V6_2);

    \draw[-, very thick] (V1_1) to (V2_1);
    \draw[-, very thick] (V2_1) to (V3_1);
    \draw[-, very thick] (V3_1) to (V4_1);
    \draw[-, very thick] (V4_1) to (V5_1);
    \draw[-, very thick] (V5_1) to (V6_1);
    \draw[-, very thick] (V6_1) to (V1_1);

    \draw[-, very thick] (V1_2) to (V2_2);
    \draw[-, very thick] (V2_2) to (V3_2);
    \draw[-, very thick] (V3_2) to (V4_2);
    \draw[-, very thick] (V4_2) to (V5_2);
    \draw[-, very thick] (V5_2) to (V6_2);
    \draw[-, very thick] (V6_2) to (V1_2);
\end{tikzpicture}
        \caption{}
        \label{fig:aug_a}
\end{subfigure}
    \begin{subfigure}{.45\textwidth}
       \centering
    \begin{tikzpicture}[vertexBall, edgeDouble, faceStyle, scale=2]

    \tikzset{knoten/.style={circle,fill=black,inner sep=0.6mm}}
    \node [knoten] (V1_1) at (1,0) {};
    \node [knoten] (V2_1) at (0.4999999999999999, 0.8660254037844386) {};
    \node [knoten] (V3_1) at (-0.5, 0.8660254037844388) {};
    \node [knoten] (V4_1) at (-1, 0.) {};
    \node [knoten] (V5_1) at (-.5, -0.8660254037844386) {};
    \node [knoten] (V6_1) at (0.5000000000000001, -0.8660254037844386) {};

    \node [knoten] (V1_2) at (0.5*1,0) {};
    \node [knoten] (V2_2) at (0.5*0.4999999999999999, 0.5*0.8660254037844386) {};
    \node [knoten] (V3_2) at (0.5*-0.5, 0.5*0.8660254037844388) {};
    \node [knoten] (V4_2) at (0.5*-1, 0.) {};
    \node [knoten] (V5_2) at (0.5*-.5, 0.5*-0.8660254037844386) {};
    \node [knoten] (V6_2) at (0.5*0.5000000000000001, 0.5*-0.8660254037844386) {};

    \node [knoten] (V) at (0, 0.) {};

    \draw[-, very thick] (V1_1) to (V1_2);
    \draw[-, very thick] (V2_1) to (V2_2);
    \draw[-, very thick] (V3_1) to (V3_2);
 \draw[-, very thick] (V4_1) to (V4_2);
    \draw[-, very thick] (V5_1) to (V5_2);
    \draw[-, very thick] (V6_1) to (V6_2);

    \draw[-, very thick] (V1_1) to (V2_1);
    \draw[-, very thick] (V2_1) to (V3_1);
    \draw[-, very thick] (V3_1) to (V4_1);
    \draw[-, very thick] (V4_1) to (V5_1);
    \draw[-, very thick] (V5_1) to (V6_1);
    \draw[-, very thick] (V6_1) to (V1_1);

    \draw[-, very thick] (V1_2) to (V2_2);
    \draw[-, very thick] (V2_2) to (V3_2);
    \draw[-, very thick] (V3_2) to (V4_2);
    \draw[-, very thick] (V4_2) to (V5_2);
    \draw[-, very thick] (V5_2) to (V6_2);
    \draw[-, very thick] (V6_2) to (V1_2);

    \draw[-, very thick] (V) to (V1_2);
    \draw[-, very thick] (V) to (V2_2);
    \draw[-, very thick] (V) to (V3_2);
    \draw[-, very thick] (V) to (V4_2);
    \draw[-, very thick] (V) to (V5_2);
    \draw[-, very thick] (V) to (V6_2);

\end{tikzpicture}
        \caption{}
        \label{fig:aug_b}
\end{subfigure}

    \caption{A 3-connected planar graph $G$ (a) and the graph that results from augmenting the inner face of length 6 of $G$ (b)}
\end{figure}

Next, we use \Cref{def:augmentation} to introduce the construction of a graph resulting by augmenting every face of a given $3$-connected planar graph.
\begin{definition}\label{def:completeaug}
Let $G$ be a $3$-connected planar graph. The graph obtained by augmenting each face of $G$ is denoted by $G^a$ and referred to as the \textbf{\emph{complete augmentation}} of $G.$
\end{definition}

\Cref{fig:complete_aug_b} illustrates the complete augmentation of the cubical graph shown in \Cref{fig:complete_aug_a}.
\begin{figure}[H]
    \begin{subfigure}{.45\textwidth}
       \centering
    \begin{tikzpicture}[vertexBall, edgeDouble, faceStyle, scale=1]

    \tikzset{knoten/.style={circle,fill=black,inner sep=0.6mm}}
    \node [knoten] (V1_1) at (-1,1) {};
    \node [knoten] (V2_1) at (1,1) {};
    \node [knoten] (V3_1) at (1,-1) {};
    \node [knoten] (V4_1) at (-1, -1.) {};
    \node [knoten] (V1_2) at (-2,2) {};
    \node [knoten] (V2_2) at (2,2) {};
    \node [knoten] (V3_2) at (2,-2) {};
    \node [knoten] (V4_2) at (-2, -2.) {};

    \draw[-, very thick] (V1_1) to (V1_2);
    \draw[-, very thick] (V2_1) to (V2_2);
        \draw[-, very thick] (V3_1) to (V3_2);
        \draw[-, very thick] (V4_1) to (V4_2);
        \draw[-, very thick] (V1_1) to (V1_2);
        
        \draw[-, very thick] (V1_1) to (V2_1);
        \draw[-, very thick] (V2_1) to (V3_1);
        \draw[-, very thick] (V3_1) to (V4_1);
        \draw[-, very thick] (V1_1) to (V4_1);

        \draw[-, very thick] (V1_2) to (V2_2);
        \draw[-, very thick] (V2_2) to (V3_2);
        \draw[-, very thick] (V3_2) to (V4_2);
        \draw[-, very thick] (V1_2) to (V4_2);

\end{tikzpicture}
        \caption{}
        \label{fig:complete_aug_a}
\end{subfigure}
    \begin{subfigure}{.45\textwidth}
       \centering
    \begin{tikzpicture}[vertexBall, edgeDouble, faceStyle, scale=1]

    \tikzset{knoten/.style={circle,fill=black,inner sep=0.6mm}}
    \node [knoten] (V1_1) at (-1,1) {};
    \node [knoten] (V2_1) at (1,1) {};
    \node [knoten] (V3_1) at (1,-1) {};
    \node [knoten] (V4_1) at (-1, -1.) {};
    \node [knoten] (V1_2) at (-2,2) {};
    \node [knoten] (V2_2) at (2,2) {};
    \node [knoten] (V3_2) at (2,-2) {};
    \node [knoten] (V4_2) at (-2, -2.) {};
    
        \node [knoten] (F1) at (0,0) {};
        \node [knoten] (F2) at (1.5,0) {};
        \node [knoten] (F3) at (-1.5,0) {};
        \node [knoten] (F4) at (0,1.5) {};

        \node [knoten] (F5) at (0,-1.5) {};

        \node [knoten] (F6) at (0,4) {};
        
    \draw[-, very thick] (V1_1) to (V1_2);
    \draw[-, very thick] (V2_1) to (V2_2);
        \draw[-, very thick] (V3_1) to (V3_2);
        \draw[-, very thick] (V4_1) to (V4_2);
        \draw[-, very thick] (V1_1) to (V1_2);
        
        \draw[-, very thick] (V1_1) to (V2_1);
        \draw[-, very thick] (V2_1) to (V3_1);
        \draw[-, very thick] (V3_1) to (V4_1);
        \draw[-, very thick] (V1_1) to (V4_1);

        \draw[-, very thick] (V1_2) to (V2_2);
        \draw[-, very thick] (V2_2) to (V3_2);
        \draw[-, very thick] (V3_2) to (V4_2);
        \draw[-, very thick] (V1_2) to (V4_2);

        \draw[-, very thick] (F1) to (V2_1);
        \draw[-, very thick] (F1) to (V3_1);
        \draw[-, very thick] (F1) to (V4_1);
        \draw[-, very thick] (F1) to (V1_1);

        \draw[-, very thick] (F2) to (V2_1);
        \draw[-, very thick] (F2) to (V3_1);
        \draw[-, very thick] (F2) to (V2_2);
        \draw[-, very thick] (F2) to (V3_2);

        \draw[-, very thick] (F3) to (V1_1);
        \draw[-, very thick] (F3) to (V4_1);
        \draw[-, very thick] (F3) to (V1_2);
        \draw[-, very thick] (F3) to (V4_2);

        \draw[-, very thick] (F4) to (V1_1);
        \draw[-, very thick] (F4) to (V2_1);
        \draw[-, very thick] (F4) to (V1_2);
        \draw[-, very thick] (F4) to (V2_2);

        \draw[-, very thick] (F5) to (V3_1);
        \draw[-, very thick] (F5) to (V4_1);
        \draw[-, very thick] (F5) to (V3_2);
        \draw[-, very thick] (F5) to (V4_2);

        \draw[-, very thick] (F6) to (V1_2);
        \draw[-, very thick] (F6) to (V2_2);

\draw[thick] (-2,-2) .. controls (-3.5,2) .. (0,4);

\draw[thick] (2,-2) .. controls (3.5,2) .. (0,4);
\end{tikzpicture}
        \caption{}
        \label{fig:complete_aug_b}
\end{subfigure}

    \caption{The cubical graph $G$ (a) and its complete augmentation $G^a$}
\end{figure}
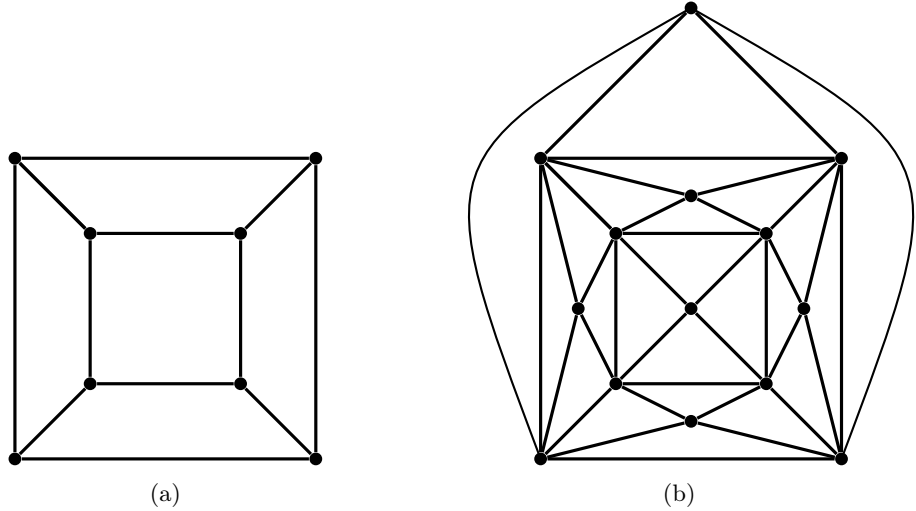

We observe that the complete augmentation of a $3$-connected planar graph $G$ is a planar graph containing exactly $\vert V(G)\vert +\vert F(G)\vert$ vertices, $\vert E(G)\vert+\sum_{F\in F(G)}\vert F\vert$ edges and $\sum_{F\in F(G)}\vert F\vert$ faces, where $\vert F\vert$ defines the length of the cycle that bounds $F$. Moreover, we notice that $G^a$ forms a maximal planar graph. 
\begin{lemma}
    The complete augmentation of a 3-connected planar graph forms a planar triangulation.
\end{lemma}
\begin{proof}
    Let $G$ be a 3-connected planar graph and let $F_1,\dots,F_\ell$ be exactly the faces of its unique planar embedding. Consider the face $F_i$ with $F_i=(v_1,\dots,v_k)$.
    By augmenting the face $F_i$, it is subdivided into the triangles $(v_F,v_i,v_{i+1})$ for $i\in\{1,\dots,k\}$ with $v_{k+1}=v_1.$ The augmentation of a face is a local operation and therefore does not affect other faces of the given graph $G$. In particular it remains planar. Thus, augmenting all faces of $G$ results in $G^a$ being a planar triangulation.
\end{proof}
Note that planar triangulations are always 3-connected. So the above lemma shows that if $G$ is 3-connected and planar, then these properties also hold for $G^a$.

\subsection{Truncation}
Next, we introduce the truncation of a given $3$-connected planar graph at one of its vertices, see for example \cite{MR370327} for more details.
\begin{definition}\label{def:truncation}
    Let $G$ be a $3$-connected planar graph, $\rho$ the rotation system of $G$ and $v\in V(G)$ a vertex with $k:=\deg(v)$.  Moreover, let $e_1=\{v,v_1\},\ldots,e_k=\{v,v_k\}\in E(G)$ be exactly the edges that are incident to $v$ and satisfy $\rho_v(e_i)=e_{i+1}$ with $e_{k+1}=e_1$, i.e.\ the edge $e_{i+1}$ is the successor of the edge $e_i$ at $v$ with respect to $\rho$. This allows us to define the graph $G'$ via
    \begin{align*}
        V(G'):=&(V(G)\setminus \{v\})\,\dot{\cup}\, \{w_{(v,e_1)},\ldots,w_{(v,e_k)}\},\\
        E(G'):=&\bigcup_{i=1}^{k}\{\{w_{(v,e_i)},w_{(v,e_{i+1})}\},\{v_{i},w_{(v,e_{i})}\}\} \cup (E(G)\setminus \{e_1,\ldots,e_k\})
    \end{align*}
    with $w_{(v,e_{k+1})}:=w_{(v,e_1)}$ and $w_{(v,e_j)}\notin V(G)$ for $1\leq j\leq k$.
    We say that $G'$ is obtained from applying a \textbf{\emph{truncation}} to the vertex $v$ in $G$. 
\end{definition}

The graph that results from a truncation of the 3-connected planar graph depicted in \Cref{subfig:truncation_a} at the vertex of degree $6$ is shown in \Cref{subfig:truncation_b}.
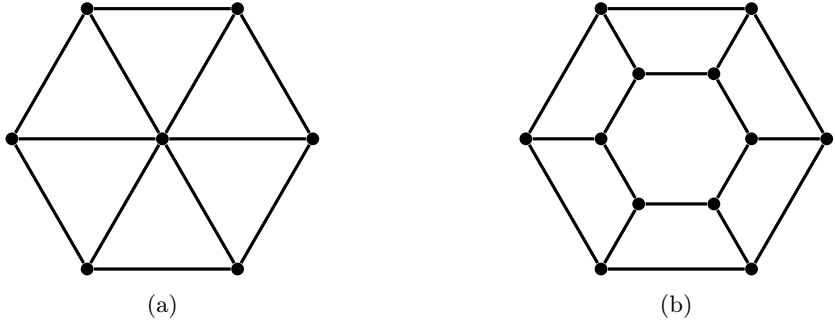
\begin{figure}[H]
\begin{subfigure}{.45\textwidth}
   \centering
\begin{tikzpicture}[vertexBall, edgeDouble, faceStyle, scale=2]
    \tikzset{knoten/.style={circle,fill=black,inner sep=0.6mm}}
    \node [knoten] (V1_3) at (1-3,0) {};
    \node [knoten] (V2_3) at (0.4999999999999999-3, 0.8660254037844386) {};
    \node [knoten] (V3_3) at (-0.5-3, 0.8660254037844388) {};
    \node [knoten] (V4_3) at (-1-3, 0.) {};
    \node [knoten] (V5_3) at (-.5-3, -0.8660254037844386) {};
    \node [knoten] (V6_3) at (0.5000000000000001-3, -0.8660254037844386) {};
    \node [knoten] (V7_3) at (0.-3, 0.) {};

    \draw[-, very thick] (V1_3) to (V2_3);
    \draw[-, very thick] (V2_3) to (V3_3);
    \draw[-, very thick] (V3_3) to (V4_3);
    \draw[-, very thick] (V4_3) to (V5_3);
    \draw[-, very thick] (V5_3) to (V6_3);
    \draw[-, very thick] (V6_3) to (V1_3);

    \draw[-, very thick] (V7_3) to (V1_3);
    \draw[-, very thick] (V7_3) to (V2_3);
    \draw[-, very thick] (V7_3) to (V3_3);
    \draw[-, very thick] (V7_3) to (V4_3);
    \draw[-, very thick] (V7_3) to (V5_3);
    \draw[-, very thick] (V7_3) to (V6_3);

    \end{tikzpicture}
    \subcaption{}
    \label{subfig:truncation_a}
    \end{subfigure}
    \begin{subfigure}{.45\textwidth}
       \centering
    \begin{tikzpicture}[vertexBall, edgeDouble, faceStyle, scale=2]

    \tikzset{knoten/.style={circle,fill=black,inner sep=0.6mm}}
    \node [knoten] (V1_1) at (1,0) {};
    \node [knoten] (V2_1) at (0.4999999999999999, 0.8660254037844386) {};
    \node [knoten] (V3_1) at (-0.5, 0.8660254037844388) {};
    \node [knoten] (V4_1) at (-1, 0.) {};
    \node [knoten] (V5_1) at (-.5, -0.8660254037844386) {};
    \node [knoten] (V6_1) at (0.5000000000000001, -0.8660254037844386) {};

    \node [knoten] (V1_2) at (0.5*1,0) {};
    \node [knoten] (V2_2) at (0.5*0.4999999999999999, 0.5*0.8660254037844386) {};
    \node [knoten] (V3_2) at (0.5*-0.5, 0.5*0.8660254037844388) {};
    \node [knoten] (V4_2) at (0.5*-1, 0.) {};
    \node [knoten] (V5_2) at (0.5*-.5, 0.5*-0.8660254037844386) {};
    \node [knoten] (V6_2) at (0.5*0.5000000000000001, 0.5*-0.8660254037844386) {};

    \draw[-, very thick] (V1_1) to (V1_2);
    \draw[-, very thick] (V2_1) to (V2_2);
    \draw[-, very thick] (V3_1) to (V3_2);
 \draw[-, very thick] (V4_1) to (V4_2);
    \draw[-, very thick] (V5_1) to (V5_2);
    \draw[-, very thick] (V6_1) to (V6_2);

    \draw[-, very thick] (V1_1) to (V2_1);
    \draw[-, very thick] (V2_1) to (V3_1);
    \draw[-, very thick] (V3_1) to (V4_1);
    \draw[-, very thick] (V4_1) to (V5_1);
    \draw[-, very thick] (V5_1) to (V6_1);
    \draw[-, very thick] (V6_1) to (V1_1);

    \draw[-, very thick] (V1_2) to (V2_2);
    \draw[-, very thick] (V2_2) to (V3_2);
    \draw[-, very thick] (V3_2) to (V4_2);
    \draw[-, very thick] (V4_2) to (V5_2);
    \draw[-, very thick] (V5_2) to (V6_2);
    \draw[-, very thick] (V6_2) to (V1_2);
\end{tikzpicture}
        \caption{}
        \label{subfig:truncation_b}
\end{subfigure}

    \caption{A $3$-connected planar graph $G$ (a) and the graph resulting from applying a truncation to $G$ at the vertex of degree 6 (b)}
\end{figure}
We observe that there is an injective map from $V(G)\setminus \{v\}$ to the vertices of the graph $G'$ resulting from a truncation of $G$ at $v$.
The above operation can be exploited to present the complete truncation of a $3$-connected planar graph. 
\begin{definition}\label{def:completetruncation}
Let $G$ be a $3$-connected planar graph with $V(G)=\{v_1,\ldots,v_n\}.$
First, we define $G_0:=G$ and recursively construct graphs $G_1,\ldots,G_n$ as follows:
For $0\leq i\leq n-1$ let $G_i$ already be constructed. 
The vertex $v_{i+1}$ can be interpreted as a vertex of the graph $G_i$.
Thus, we define the graph $G_{i+1}$ as the graph that results from $G_i$ by applying a truncation at the vertex $v_{i+1}$.
We call $G^t:= G_n$ the \textbf{\emph{complete truncation}} of $G.$
\end{definition}
In \Cref{fig:ct}, we illustrate the graph $(K_4)^t$ resulting from the complete truncation of the graph $K_4$ (shown in \Cref{fig:k4}).

Applying the complete truncation to a $3$-connected planar graph $G$ results in a graph with $\sum_{v\in V(G)} \deg(v)$ vertices, $\vert E(G)\vert+\sum_{v\in V(G)} \deg(v)$ edges and $\vert V(G)\vert + \vert F(G)\vert$ faces.
The complete truncation $G^t$ contains two different types of edges. We can use these edges to establish a one-to-one correspondence between the vertices and edges of $G$ and certain structures of $G^t.$ This observation is summarised in the following remark.

\begin{remark}\label{remark:completeTrunc}
Let $G$ be a $3$-connected planar graph, $\rho$ its (unique) rotation system and $G^t$ its complete truncation. Thus, the vertices of $G^t$ are given by $$V(G^t)=\{w_{(v,e)}\mid v\in V(G), e\in E(G) \text{ with } v\in e \}.$$
As mentioned above, $G^t$ has two types of edges, namely 
\begin{itemize}
    \item[1.] $\vert E(G) \vert $ edges of the form $\{w_{(v,e)},w_{(v',e)}\},$ where $e\in E(G)$ is an edge with $e=\{v,v'\},$ and 
    \item[2.] $\sum_{v\in V(G)} \deg(v)$ edges of the form $\{w_{(v,e)},w_{(v,e')}\},$ where $e,e'\in E(G)$ are edges with $v\in e \cap e'$ and $\rho_{v} (e)=e'.$
\end{itemize}
We therefore see that an edge $e=\{v,v'\}\in E(G)$ corresponds to the edge $\{w_{(v,e)},w_{(v',e)}\}\in E(G^t)$. Moreover, a vertex $v\in V(G)$ whose incident edges are given by $e_1,\ldots,e_m \in E(G)$ with $\rho_{v}(e_i)=e_{i+1}$ for $i=1,\ldots m$ with $e_{m+1}:=e_1$ corresponds to the following face:
\[
(w_{(v,e_1)},\ldots,w_{v_{(v,e_m)}})\in F(G^t).
\]
Note that the remaining $\vert F(G) \vert $ faces of $G^t$ are in a one-to-one correspondence with the faces of $G.$
\end{remark}

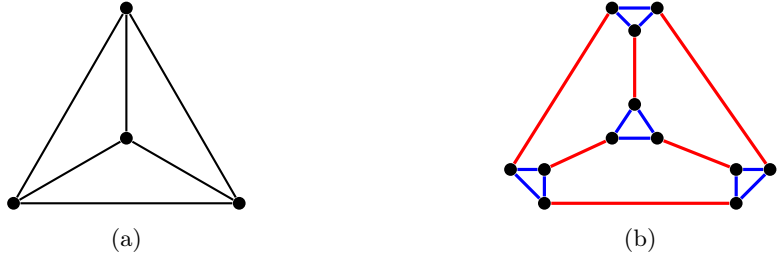
\begin{figure}[H]
    \begin{subfigure}{.45\textwidth}
        \centering
        \begin{tikzpicture}[scale=3]
            \tikzset{knoten/.style={circle,fill=black,inner sep=0.6mm}}
            \node [knoten] (a) at (0.,0) {};
            \node [knoten] (b) at (1,0) {};
            \node [knoten] (c) at (0.5,0.8660) {};
            \node [knoten] (d) at (0.5,0.33333*0.8660) {};

            \draw[-,thick] (a) to (b);
            \draw[-,thick] (a) to (c);
            \draw[-,thick] (a) to (d);
            \draw[-,thick] (b) to (c);
            \draw[-,thick] (b) to (d);
            \draw[-,thick] (c) to (d);
        
        \end{tikzpicture}
        \caption{}
        \label{fig:k4}
    \end{subfigure}
 \begin{subfigure}{.45\textwidth}
        \centering
        \begin{tikzpicture}[scale=3]
            \tikzset{knoten/.style={circle,fill=black,inner sep=0.6mm}}
            \node [knoten] (a1) at (0.1,0) {};
            \node [knoten] (a2) at (0.1,0.15) {};
            \node [knoten] (a3) at (-0.05,0.15) {};
            \node [knoten] (b1) at (1.1,0.15) {};
            \node [knoten] (b2) at (0.95,0) {};
            \node [knoten] (b3) at (0.95,0.15) {};
            \node [knoten] (c1) at (0.5,0.8660-0.1) {};
            \node [knoten] (c2) at (0.4,0.8660) {};
            \node [knoten] (c3) at (0.6,0.8660) {};
            \node [knoten] (d1) at (0.4,0.33333*0.8660) {};
            \node [knoten] (d2) at (0.6,0.33333*0.8660) {};
            \node [knoten] (d3) at (0.5,0.33333*0.8660+0.15) {};
            
            \draw[-,very thick,red] (a1) to (b2);
            \draw[-,very thick,red] (a3) to (c2);
            \draw[-,very thick,red] (a2) to (d1);
            \draw[-,very thick,red] (b1) to (c3);
            \draw[-,very thick,red] (b3) to (d2);
            \draw[-,very thick,red] (c1) to (d3);
            
            \draw[-,very thick,blue] (c1) to (c2);
            \draw[-,very thick,blue] (c2) to (c3);
            \draw[-,very thick,blue] (c1) to (c3);
            
            \draw[-,very thick,blue] (a1) to (a2);
            \draw[-,very thick,blue] (a2) to (a3);
            \draw[-,very thick,blue] (a1) to (a3);
            
            \draw[-,very thick,blue] (b1) to (b2);
            \draw[-,very thick,blue] (b2) to (b3);
            \draw[-,very thick,blue] (b1) to (b3);

            \draw[-,very thick,blue] (d1) to (d2);
            \draw[-,very thick,blue] (d2) to (d3);
            \draw[-,very thick,blue] (d1) to (d3);
        \end{tikzpicture}
        \caption{}
        \label{fig:ct}
    \end{subfigure}
    \caption{The complete graph $G\cong K_4$ (a) and its complete truncation $G^t$ where the edges of Type~1 are coloured red and the edges of Type~2 are coloured blue (b)}
    \label{fig:exApollonian}
\end{figure}

We conclude this section by stating that the complete truncation $G^t$ is a $3$-connected cubic planar graph. This result follows from properties of convex polyhedra under truncation, as discussed in \cite{conway}.

\begin{lemma}\label{remark:complete}
    Let $G$ be a 3-connected planar graph, then the complete truncation $G^t$ forms a 3-connected cubic planar graph.
\end{lemma}
\begin{proof}
    It is easy to see that $G^t$ is cubic and planar. Hence, it remains to show that $G^t$ is 3-connected. By Steinitz Theorem, see \cite{Steinitz}, we know that $G$ is the graph describing the incidences between the vertices and edges of a convex polyhedron $P$.
    Truncating all vertices of $P$ results in a convex polyhedron where the incidences between the vertices and edges is described by $G^t$. Thus, Steinitz Theorem implies that $G^t$ is 3-connected and planar.
\end{proof}

\section{Orientable circuit double covers of $3$-connected planar graphs}\label{sec:orientable}

Finally, we investigate the class of Apollonian duals and show that these graphs are exactly the $3$-connected planar graphs with exactly one orientable circuit double cover. We need some preparation to conclude this result. We start by showing that the edges of an Apollonian network fall into two classes.
To formalise this, we introduce the inverse of augmenting a face of a given $3$-connected planar graph which is the deletion of a vertex. Given a $3$-connected planar graph $G$ and a vertex $v\in V(G)$, we denote by $G_v$ the graph obtained by removing $v$ along with all edges incident to $v$ from $G$.

\begin{proposition}\label{prop:sep3deg}
  Let $G$ be an Apollonian network and $e\in E(G)$ an edge. Then $e$ is either contained in a separating triangle of $G$ or incident to a vertex of degree three.
\end{proposition}
\begin{proof}
    We verify this statement by induction over $n:=\vert V(G)\vert\geq 4$. If $n=4,$ then $G$ is isomorphic to the complete graph $K_4$. Hence, it is easy to see that every edge in $E(G)$ is incident to a vertex of degree three. 
    Now, we assume that $n>4$ holds. Thus, $G$ contains a vertex $v\in V(G)$ with $\deg(v)=3$ such that $G_v$ is a well-defined Apollonian network. Thus, $v$ has three neighbours in $G$, namely $w_1,w_2$ and $w_3$. 
    By our inductive argument, every edge in $E(G_v)$ is either contained in a separating triangle or incident to a vertex of degree three in $G_v$. In $G$, all edges except the edges $\{w_i,w_j\}$ for $i,j\in\{1,2,3\}$ and $i\neq j$ are therefore also contained in a separating triangle or incident to a vertex of degree three in $G$.
    Moreover, the edges $\{w_i,w_j\}$ for $i,j\in\{1,2,3\}$ and $i\neq j$ form a separating triangle in $G$. Finally, the edges incident to $v$ in $G$, namely $\{v,v_1\},\{v,v_2\},\{v,v_3\}$ are incident to a vertex of degree three. Thus, the result follows.
\end{proof}
Note that this fact follows directly from the fact that these objects correspond exactly to stacked $2$-spheres, see \cite{MR4406230} for example. Nevertheless, for completeness, we include a proof in the framework of the present paper.
Next, we show that dualising, truncating and augmenting can be composed to obtain isomorphic graphs. To proceed, we require the definition of the dual graph. The \textbf{dual} of a planar graph $G$, denoted by $G^\ast$, is constructed by replacing each face with a vertex with two vertices adjacent whenever their corresponding faces share an edge. This means that the vertices of $G$ are translated to faces of $G^\ast$. We show that if $G$ is a $3$-connected planar graph, then the complete augmentation $(G^\ast)^a$ of $G^\ast$ forms a graph that is isomorphic to the dual $(G^t)^\ast$ of the complete truncation $G^t$. This statement is further illustrated in \Cref{fig:proofidea}.
This statement is classical in the setting of polyhedra (see, for example, \cite{conway}), 
and by Steinitz’s theorem \cite{Steinitz} it extends to 3‑connected planar graphs. However, in this paper, we include a proof of this statement for completeness.
\begin{figure}[H]
        \centering
        \begin{tikzpicture}[scale=1.6]
            \tikzset{knoten/.style={circle,fill=black,inner sep=0.6mm}}
            \node  (a1) at (0.3,0) {};
            \node  (a2) at (1.7,0) {};
             \node  (b1) at (0,0.3) {};
            \node  (b2) at (0,1.7) {};
             \node  (c1) at (2,0.3) {};
            \node  (c2) at (2,1.7) {};
             \node  (d1) at (0.3,2) {};
            \node  (d2) at (1.7,2) {};
            
            \draw[->,thick] (a1) to (a2);
            \draw[<-,thick] (b1) to (b2);   
            \draw[<-,thick] (c1) to (c2);
            \draw[->,thick] (d1) to (d2);   
            \node at (0,0) {\Large $G^\ast$};
            \node at (0,2) {\Large $G$};
            \node at (2,2) {\Large $G^t$};
            \node at (2,0) {\Large $G'$};
            \node at (1,-0.25) {\Large $(\cdot)^a$};
            \node at (1,2.25) {\Large $(\cdot)^t$};
            \node at (-0.25,1) {\Large $(\cdot)^\ast$};
            \node at (2.25,1) {\Large $(\cdot)^\ast$};

        \end{tikzpicture}
    \caption{Commuting diagram describing the idea of $G'\cong (G^\ast)^a\cong (G^t)^\ast $}
    \label{fig:proofidea}
\end{figure} 

\begin{theorem}\label{theorem:dualtruncatedaug}
    Let $G$ be a 3-connected planar graph. Then $(G^\ast)^a\cong (G^t)^\ast.$
\end{theorem}
\begin{proof}
We prove the result by first giving explicit descriptions of the above graphs and then establishing a suitable isomorphism $\Phi:(G^\ast)^a \to (G^t)^\ast$. Let us start by describing the incidence structure of $(G^\ast)^a$. Per definition, all faces and vertices in $G^\ast$ correspond to vertices and faces in $G$, respectively. Thus, the complete augmentation $(G^\ast)^a$ is a planar triangulation, where the vertices correspond to the faces and vertices of $G^\ast$ (and hence to vertices and faces of $G$). More precisely, $(G^\ast)^a$ is the graph given by the vertices $V=\{w_v\mid v\in V(G)\}\cup \{w_F\mid F\in F(G)\}$ and the edges 
\begin{align*}
E=&\{\{w_F,w_v\}\mid v\in F \text{ where } v\in V(G),F\in F(G)\}\, \cup \\ &\{\{w_{F},w_{F'}\}\mid F\cap F' \in E(G) \text{ where }F,F'\in F(G) \}.
\end{align*}
Now, let us consider $(G^t)^\ast.$
We know that $G^t$ is a cubic planar graph, where the faces either correspond to vertices or to faces of $G$. More precisely, for every vertex $v\in V(G)$ the graph $G^t$ contains a face $H_v$ and for every face $F\in F(G)$ it has a face $H_F.$ Note, two faces $H,H'\in F(G^t)$ satisfy $H\cap H'\in E(G^t)$ if and only if one of the following two cases occurs: 
\begin{enumerate}
    \item $H=H_v$ for a vertex $v\in V(G)$ and $H'=H_F$ for a face $F\in F(G)$ with $v\in F$,
    \item $H=H_F$ for a face $F\in F(G)$ and $H'=H_F'$ for a face $F'\in F(G)$ with $F\cap F'\in E(G).$
\end{enumerate}
Hence, $(G^t)^\ast=(V',E')$ is the graph with $V'=\{u_{v}\mid v\in V(G)\}\cup \{u_F\mid F\in F(G)\}$ and \begin{align*}
E'=&\{\{u_F,u_v\}\mid v\in F \text{ where } v\in V(G),F\in F(G)\}\,\cup \\ &\{\{u_{F},w_{F'}\}\mid F\cap F' \in E(G) \text{ where }F,F'\in F(G) \}.
\end{align*}
Therefore, it is easy to see that the map 
$$\Phi: (G^\ast)^a\to (G^t)^\ast, w\mapsto  
\begin{cases}
  u_v, & w=w_v \text{ for a vertex }v\in V(G),\\
u_F, & w=w_F \text{ for a face }F\in F(G).\\
\end{cases}$$ forms an isomorphism from $(G^\ast)^a$ to $(G^t)^\ast$ which concludes the result.
\end{proof}

The next lemma turns out to be really helpful as it states that orientable circuit double covers of the complete truncation of a 3-connected planar graph $G$ can be translated into orientable circuit double covers of $G$.
\begin{lemma}\label{lemma:transCDC}
    Let $G$ be a 3-connected planar graph. If $G^t$ has an orientable circuit double cover, then $G$ has an orientable circuit double cover.
\end{lemma}

\begin{proof}
   Recall that the vertices and edges of $G^t$ are defined by the following sets:
    $$V(G^t)=\{w_{(v,e)}\mid v\in V(G), e\in E(G), v\in e \}\text{ and } E(G^t)=E_1 \cup\, E_2 \text{ with}\\ $$
    \begin{align*}
        E_1&= \{\{w_{(v,e)},w_{(v,e')}\} \mid e,e'\in E(G),\ v\in e \cap e',\ \rho_{v}(e)=e'\} \text{ and } \\
        E_2&= \{\{w_{(v,e)},w_{(v',e)}\} \mid e\in E(G),\ e=\{v,v'\}\}.
    \end{align*}
The edges in $E_2$ are in bijection to $E(G)$ and a face $F$ of $G^t$ corresponds to a vertex of $G$ or to a face of $G$, as explained in \Cref{remark:completeTrunc}.  
    Let $\mathcal{C}^t$ be an orientable circuit double cover of $G^t$. First, we show that $\mathcal{C}^t$ can be translated to a circuit double cover $\mathcal{C}$ of $G$. In $G^t$ a cycle $C$ can be translated to a circuit in $G$ by the following rules:
    \begin{itemize}
        \item[1.] If $C\in \mathcal{C}^t$ is a face of $G^t$ that corresponds to a vertex of $G$, then $C$ is ignored for the circuit double cover of $G$.
        \item[2.] If $C\in \mathcal{C}^t$ is a face of $G^t$ that corresponds to a face of $G$, then $C$ can be translated directly into a face of $G$ as mentioned in \Cref{remark:completeTrunc}.
        \item[3.] If $C\in \mathcal{C}^t$ traverses different faces of $G^t$, then $C$ can be translated to circuits of $G$ by removing the edges contained in $E_1$. The remaining edges, i.e.\ the edges in $E_2$, can be translated into edges in $G$ that form a cycle in $G$ as described in \Cref{remark:completeTrunc}.
    \end{itemize}
    So, by following the above rules the set $\mathcal{C}^t$ can be transformed into a set of circuits $\mathcal{C}$ of $G$. Since each edge $e\in E_2$ is covered by two different cycles in $\mathcal{C}^t$ and these cycles are translated into different circuits in $\mathcal{C}$ that cover the edge $e'\in E(G)$ corresponding to $e$, we see that $\mathcal{C}$ forms a circuit double cover of $G.$
    It remains to show that $\mathcal{C}$ is again orientable. Note that the cycles of $\mathcal{C}^t$ are only shortened to obtain $\mathcal{C}$ and nothing else is changed. By this shortening, only the edges in bijection to the edges in $E(G)$ remain. Via this bijection the orientation of the cycles in $\mathcal{C}^t$ can be translated directly to a suitable orientation of the circuits in $\mathcal{C}$. 
\end{proof}

The rules established in the proof of \Cref{lemma:transCDC} allow us to translate different cycle double covers of $G^t$ into different circuit double covers of $G$. 
Hence, we obtain the following corollary.
\begin{corollary}\label{cor:numberscdc}
Let $G$ be a 3-connected planar graph. If $c_G$ and $c_{G^t}$ denote the numbers of circuit and cycle double covers of $G$ and $G^t$, respectively, then $c_G\geq c_{G^t}$ holds. 
\end{corollary}

Finally, we have introduced all required tools and are able to prove \Cref{theorem:main}.

\begin{proof}
We start this proof by assuming $G$ to be an Apollonian dual. This means that $G$ is cubic and so every circuit double cover of $G$ is a cycle double cover. Thus, we know that $G$ has exactly one orientable circuit double cover by \cite[Theorem 5.4]{UnserPaper}.
     
To conclude the result, let us now assume that the other direction does not hold. This implies the existence of a $3$-connected planar graph $G$ that is not isomorphic to an Apollonian dual and that has exactly one orientable circuit double cover. Note that $G$ cannot be cubic because of~\cite[Theorem 5.4]{UnserPaper}. Furthermore, we notice that different orientable cycle double covers of the complete truncation $G^t$ correspond to different orientable circuit double covers of $G$. By \Cref{cor:numberscdc} $G^t$ has to be a $3$-connected cubic planar graph with exactly one orientable cycle double cover, i.e.\ an Apollonian dual by \cite{UnserPaper}.
Consequently, $(G^t)^\ast$ is an Apollonian network. By \Cref{theorem:dualtruncatedaug}, the graph $(G^\ast)^a$ is isomorphic to $(G^t)^\ast$, and therefore $(G^\ast)^a$ is an Apollonian network. In the following we show that applying a complete augmentation to $G^\ast$ cannot result in an Apollonian network.

Since $G$ is not cubic, $G^\ast$ contains a face $F = (v_1,\ldots,v_n)$ of length $n \geq 4$. Therefore, $(G^\ast)^a$ contains a vertex $v_F$ of degree at least four resulting from the augmentation of $G^\ast$ at the face $F$. Note, the vertices $v_1,\dots,v_n$ have degree at least six in $(G^\ast)^a$ since they have degree at least three in $G^\ast$. This means that the edges $\{v_F,v_i\}$ for $i\in\{1,\dots,n\}$ are not incident in $G^\ast$ to a vertex of degree three.
Thus, by \Cref{prop:sep3deg} the edges $\{v_F,v_i\}$ for $i\in\{1,\dots,n\}$ have to be contained in a separating triangle in order for $(G^\ast)^a$ to be an Apollonian network.
This means that for every $1 \leq i \leq n$, there is an index $s_i$ (not necessarily unique) with $1 \leq s_i \leq n$ such that $(v_F, v_i, v_{s_i})$ is a separating triangle of $(G^\ast)^a$, implying $\lvert i - s_i \rvert > 1$.
Next, we argue that for every $1 \leq i \leq n$ there exists an integer $k$ with $\min(i,s_i) < k < \max(i,s_i)$ such that $(v_F, v_k, v_{s_k})$ forms a separating triangle of $(G^\ast)^a$, where either $s_k < \min(i,s_i)$ or $s_k > \max(i,s_i)$.
We assume that this statement does not hold.
Hence, there exists some $1 \leq i \leq n$ such that for every $\min(i,s_i) < k < \max(i,s_i)$, there is an index $s_k$ with $\min(i,s_i) < s_k < \max(i,s_i)$ for which $(v_F, v_k, v_{s_k})$ forms a separating triangle.
Thus, for every $\ell$ with $\min(i,s_i)< \ell< \min(i,s_i)+\lfloor\frac{\vert i-s_i\vert}{2}\rfloor$ the inequality $\vert \ell-s_\ell\vert > \vert (\ell+1)-s_{\ell+1}\vert$ must be satisfied. 
Hence, $k:=\min(i,s_i)+\lfloor\frac{\vert i-s(i)\vert}{2}\rfloor$ satisfies $\vert k-s_k\vert\in\{1,2\}.$ If $\vert k-s_k\vert=1$, then $(v_F,v_k,v_{s_k})$ cannot form a separating triangle, since it is a face of $(G^\ast)^a$, see \Cref{subfig:mainTh_a}. Moreover if $\vert k-s_k\vert =2$, we know that $v_k$ and $v_{s_k}$ are incident to a vertex $v_\ell$ which cannot satisfy $\min(i,s_k)< s_\ell < \max(i,s_k)$, see \Cref{subfig:mainTh_b}. 
We therefore obtain our desired contradiction. This means that the vertices $v_F,v_1,v_{s_1},v_k,v_{s_k}$ imply a minor of $(G^\ast)^a$ isomorphic to the complete graph $K_5.$
This contradicts the planarity of $(G^\ast)^a$ and so $(G^\ast)^a$ cannot be an Apollonian network which concludes the result.
\end{proof}
\begin{figure}[H]
   \centering
    \begin{subfigure}{0.4\textwidth}
    \begin{tikzpicture}[vertexBall, edgeDouble, faceStyle, scale=2]
        \tikzset{knoten/.style={circle,fill=black,inner sep=0.6mm}}

         \draw[-,very thick] (-4,0) arc [start angle=180, end angle=0, x radius=1, y radius=1.5];
        
        \node [knoten,label=right:$v_i$] (V1_3) at (1-3,0) {};
        \node [knoten,label=$v_k$] (V2_3) at (0.4999999999999999-3, 0.8660254037844386) {};
        \node [knoten,label=$v_{s_k}$] (V3_3) at (-0.5-3, 0.8660254037844388) {};
        \node [knoten,label=left:$v_{s_i}$] (V4_3) at (-1-3, 0.) {};
        \node [knoten] (V5_3) at (-.5-3, -0.8660254037844386) {};
        \node [knoten] (V6_3) at (0.5000000000000001-3, -0.8660254037844386) {};
        \node [knoten] (V7_3) at (-3, 0.) {};
        \node at (-2.7,0.12) {$v_F$};
    
        \draw[-, very thick] (V1_3) to (V2_3);
        \draw[-, very thick] (V2_3) to (V3_3);
        \draw[-, very thick] (V3_3) to (V4_3);
        \draw[-, very thick] (V4_3) to (V5_3);
        \draw[-, very thick] (V5_3) to (V6_3);
        \draw[-, very thick] (V6_3) to (V1_3);

        \draw[-, very thick] (V7_3) to (V1_3);
        \draw[-, very thick] (V7_3) to (V2_3);
        \draw[-, very thick] (V7_3) to (V3_3);
        \draw[-, very thick] (V7_3) to (V4_3);
        \draw[-, very thick] (V7_3) to (V5_3);
        \draw[-, very thick] (V7_3) to (V6_3);

    \end{tikzpicture}
    \caption{}
    \label{subfig:mainTh_a}
    \end{subfigure}
    \begin{subfigure}{0.4\textwidth}
    \begin{tikzpicture}[vertexBall, edgeDouble, faceStyle, scale=2]
        \tikzset{knoten/.style={circle,fill=black,inner sep=0.6mm}}

         \draw[-,very thick] (-.5-3, -0.8660254037844386) arc [start angle=-140, end angle=-5, x radius=0.85, y radius=1.5];

         \draw[-,very thick] (-2.5, 0.8660254037844386) arc [start angle=40, end angle=175, x radius=0.85, y radius=1.5];
        
        \node [knoten,label=right:$v_i$] (V1_3) at (1-3,0) {};
        \node [knoten,label=$v_k$] (V2_3) at (0.4999999999999999-3, 0.8660254037844386) {};
        \node [knoten,label=$v_l$] (V3_3) at (-0.5-3, 0.8660254037844388) {};
        \node [knoten,label=left:$v_{s_k}$] (V4_3) at (-1-3, 0.) {};
        \node [knoten,label=left:$v_{s_i}$] (V5_3) at (-.5-3, -0.8660254037844386) {};
        \node [knoten] (V6_3) at (0.5000000000000001-3, -0.8660254037844386) {};
        \node [knoten] (V7_3) at (0.-3, 0.) {};

        \node at (-2.7,0.12) {$v_F$};
    
        \draw[-, very thick] (V1_3) to (V2_3);
        \draw[-, very thick] (V2_3) to (V3_3);
        \draw[-, very thick] (V3_3) to (V4_3);
        \draw[-, very thick] (V4_3) to (V5_3);
        \draw[-, very thick] (V5_3) to (V6_3);
        \draw[-, very thick] (V6_3) to (V1_3);

        \draw[-, very thick] (V7_3) to (V1_3);
        \draw[-, very thick] (V7_3) to (V2_3);
        \draw[-, very thick] (V7_3) to (V3_3);
        \draw[-, very thick] (V7_3) to (V4_3);
        \draw[-, very thick] (V7_3) to (V5_3);
        \draw[-, very thick] (V7_3) to (V6_3);

    \end{tikzpicture}
    \caption{}
    \label{subfig:mainTh_b}
    \end{subfigure}
    \caption{$\vert k-s_k\vert=1$ (a) and $\vert k-s_k\vert=2$ (b)}
    \label{fig:mainTh}
\end{figure}

\section*{Acknowledgements}
We gratefully acknowledge the funding by the Deutsche Forschungsgemeinschaft (DFG, German Research Foundation) in the framework of the Collaborative Research Centre CRC/TRR 280 “Design Strategies for Material-Minimized Carbon Reinforced Concrete Structures – Principles of a New Approach to Construction” (project ID 417002380). Furthermore, R.\ Akpanya was supported by a grant from the Simons Foundation (SFI-MPS-Infrastructure-00008650).

\bibliographystyle{plain}
\bibliography{Extension}

\end{document}